\documentclass[11pt]{amsart}

\usepackage[utf8]{inputenc}
\usepackage[english]{babel}
 
\usepackage{geometry}
\usepackage{anysize}
\usepackage{amsfonts}
\usepackage{amssymb}
\usepackage{amsmath}
\usepackage{amsthm}
\usepackage{tikz-cd}
\usepackage{amscd}
\usepackage[all]{xy}
\usepackage{enumerate}
\usepackage{multirow}
\usepackage{latexsym} 
\usepackage{comment}
\usepackage{color}
\usepackage{colonequals}
\usepackage{mathrsfs}

\usepackage{epsfig}
\usepackage{tikz}
\usepackage{tikz-cd}

\definecolor{darkgreen}{rgb}{0,0.5,0}
\usepackage[
        draft=false,
        colorlinks, citecolor=darkgreen,
        pdfauthor={Filippo F. Favale and Sonia Brivio},
        pdftitle={On vector bundles over reducible curves with a node},
        linktocpage        
]{hyperref}

\usepackage[alphabetic,backrefs,lite]{amsrefs} 
\usepackage[all]{xy} 
\usepackage{enumerate} 

\newfont{\sheaf}{eusm10 scaled\magstep1}

\newtheorem{definition}{Definition}[section]

\newtheorem{proposition}{Proposition}[section]
\newtheorem{corollary}[proposition]{Corollary}
\newtheorem{lemma}[proposition]{Lemma}
\newtheorem{theorem}[proposition]{Theorem}

\newtheorem{remark}{Remark}[proposition]

\newtheorem{notation}{Notation}[section]

\DeclareMathOperator{\PGL}{PGL}
\DeclareMathOperator{\GL}{GL}

\DeclareMathOperator{\Hom}{Hom}
\DeclareMathOperator{\Pic}{Pic}

\DeclareMathOperator{\Rk}{Rk}

\newcommand{\Z}{\mathbb{Z}}

\newcommand{\C}{\mathbb{C}}

\newcommand{\PP}{\mathbb{P}}
\newcommand{\OO}{\mathcal{O}}

\newcommand{\U}{\mathcal{U}}

\newcommand{\id}{id}

\def\BOX{\hfill\lower.5\baselineskip\hbox{$\Box$}}

\numberwithin{equation}{section}

\newcounter{nootje}
\renewcommand\check[1]
  {\marginpar{\tiny\begin{minipage}{20mm}\begin{flushleft}\thenootje : #1\end{flushleft}\end{minipage}}\addtocounter{nootje}{1}}
\begin{document}

\title[Vector bundles over reducible curves with a node]{On vector bundles over reducible curves with a node}

\author{Filippo F. Favale}
\address{Dipartimento di Matematica e Applicazioni,
	Universit\`a degli Studi di Milano-Bicocca,
	Via Roberto Cozzi, 55,
	I-20125 Milano, Italy}
\email{filippo.favale@unimib.it}

\author{Sonia Brivio}
\address{Dipartimento di Matematica e Applicazioni,
	Universit\`a degli Studi di Milano-Bicocca,
	Via Roberto Cozzi, 55,
	I-20125 Milano, Italy}
\email{sonia.brivio@unimib.it}

\date{\today}
\thanks{
Both authors are partially supported by INdAM - GNSAGA.\\
}

\begin{abstract}
Let $C$ be a  curve with two smooth components and a single node. Let $\mathcal{U}_C(r,w,\chi)$ be the moduli space of $w$-semistable classes of depth one sheaves on $C$ having rank $r$ on both components and Euler characteristic $\chi$. In this paper, under suitable assumptions, we produce a projective bundle over the product of the moduli spaces of semistable vector bundles of rank $r$ on each components and we show that it is birational to an irreducible component of $\mathcal{U}_C(r,w,\chi)$. Then we prove the rationality of the closed subset containing vector bundles with given fixed determinant.
\end{abstract}

\maketitle
\section*{Introduction}
Moduli spaces of vector bundles on curves have always been a central topic in Algebraic Geometry. The construction of moduli space of isomorphism classes of stable vector bundle of rank $r$ and degree $d$ on a smooth projective curve of   genus $g \geq 2$ is due to Mumford (see \cite{M}). Such a moduli space is a non-singular quasi-projective variety, whose compactification was  obtained  by Seshadri in  \cite{S0}, by introducing $S$-equivalence relation between semistable vector bundles, and it is denoted by ${\mathcal U}_C(r,d)$. The compactification is a normal irreducible projective variety of dimension $r^2(g-1)+1$. When $r$ and $d$ are coprime,  the notion of semistability coincides with that of stability, so ${\mathcal U}_C(r,d)$ parametrizes isomorphism classes of stable vector bundles. Moreover, in this case there exists a Poincar\'e bundle on ${\mathcal U}_C(r,d)$, see \cite{R}. If $L \in \Pic^{d}(C)$ is a line bundle, the moduli space  ${\mathcal SU}_C(r,L)$, parametrizing semistable vector bundles of rank $r$ and fixed determinat $L$, is also of great interest. Indeed, up to a finite \'etale covering, the moduli space ${\mathcal U}_C(r,d)$ is isomorphic to the product of ${\mathcal SU}_C(r,L)$ and $\Pic^0(C)$. Hence, a lot of the geometry of ${\mathcal U}_C(r,d)$ is encoded in ${\mathcal SU}_C(r,L)$. Moreover, ${\mathcal SU}_C(r,L)$ is interesting on its own and it is a rational variety when $r$ and $d$ are coprime, see \cite{KS}. The geometry of these moduli spaces has  been studied by many authors, in particular its  relation with generalized theta functions, see \cite{B} for a survey, (see \cite{BF1}, \cite{B1}, \cite{B2},  \cite{B3}, \cite{BB} and \cite{BV}   for recent works by the authors). 

Unfortunately, as soon as the base curve becomes singular, the above results do not apply anymore. For example, for a singular irreducible curve,  in order to have a compact moduli space one possible approach consists in considering torsion-free sheaves instead of locally
free, see \cite{N} and \cite{S}. 
This method was generalized for a reducible (but reduced) curve by Seshadri.
Its idea was to include in the moduli space also depth one sheaves and to  introduce the notion of polarization $w$ and of $w$-semistability. More precisely, we denote by ${\mathcal U}_C(w,r,\chi)$ the moduli space parametrizing $w$-semistable sheaves of depth one  of rank $r$  on each  components and Euler characteristic $\chi$.

In this paper we will assume that  $C$ is  a nodal reducible curve  with two smooth irreducible components $C_1$ and $C_2$, of genera $g_i \geq 1$ and a single node $p$. We can get the curve by gluing $C_1$ and $C_2$ at the points $q_1$ and $q_2$.  
In this hypothesis, the moduli space ${\mathcal U}_C(w,r,\chi)$  is a connected reducible projective variety, see \cite{T1} and \cite{T2}; each irreducible component has dimension $r^2(p_a(C) -1)+1$ and it  corresponds to a possible pair of multidegree, see Section \ref{sec2} for details. For  problems about the stability of Kernel bundles on such curves the reader can see \cite{BF2}. 

In the above hypothesis, for any $r \geq 2$, fix a pair of integers $(d_1,d_2)$ which are  both  coprime with $r$.
The existence of Poincar\'e  vector bundles, on the moduli spaces ${\mathcal U}_{C_i}(r,d_i)$, allows us to 
produce a  projective bundle $\pi:{\mathbb P}(\mathcal F)\to {\mathcal U}_{C_1}(r,d_1) \times  {\mathcal U}_{C_2}(r,d_2)$, whose fiber 
at $([E_1],[E_2])$ is ${\mathbb P}(\Hom(E_{1,q_1},E_{2,q_2}))$,  see Lemma \ref{LEM:projbundle}.  
Let  $u \in {\mathbb P}(\mathcal F)$, $u = ((E_1],[E_2]), [\sigma])$, where $\sigma$ is a non zero homomorphism  $ E_{1,q_1} \to E_{2,q_2}$. We can associate to $u$ a depth one sheaf $E_u$ on the curve $C$, which is obtained, roughly speaking, by gluing  $E_1$ and $E_2$ along the fibers at $q_1$ and $q_2$ with $\sigma$. This is a vector bundle if and only if $\sigma$ is an isomorphism. Our first concern is to study when $E_u$ turns out to be $w$-semistable for some polarization $w$: we are able to give some necessary and sufficient conditions to ensure $w$-semistability (see section \ref{sec3}).
Then we turn our attention to the rational map
$$
\xymatrix{
\phi:{\mathbb P}({\mathcal F}) \ar@{-->}[r]&  {\mathcal U}_C(w,r,\chi)
}
$$
sending $u$ to $E_u$. Our first result (Theorem \ref{THM:MAIN}) can be summarized  in the following statement:

{\bf Theorem A}
{\it Let $C$ be a reducible nodal curve as above.
Let  $r \geq 2$  and $d_1$ and $d_2$ be integer coprime with $r$. Set $\chi_i = d_i + r(1-g_i)$ and
$\chi = \chi_1 + \chi_2 -r$.
For any pair $(\chi_1,\chi_2)$ in a suitable non empty subset of $\mathbb{Z}^2$ 
there exists a polarization $w$ such that 
${\mathbb P}({\mathcal F})$ is birational to the irreducible component of the moduli space
${\mathcal U}_C(w,r,\chi)$ corresponding to bidegree $(d_1,d_2)$.}

The birational map of the statement is the map $\phi$. We prove that it is  an  injective morphism on the open  subset  ${\mathscr U}  \subset {\mathbb P}(\mathcal F) $, given by  points $u$ where $\sigma$ is an isomorphism.  The image $\phi(\mathscr{U})$ is a dense subset of the moduli space and its points are classes of 
vector bundles whose restriction to each component is stable (see Theorem \ref{THM:MAIN}). 

Moreover, when $g_i > r+ 1$, we can give some more information about the domain of $\phi$  as follows, see
Theorem \ref{THM:2}.

{\bf Theorem B}
{\it Assume that the hypothesis of  Theorem A hold. 
If  $g_i > r +1$, 
for any pair $(\chi_1,\chi_2)$ in a suitable non empty subset of $\mathbb{Z}^2$ there exists a non empty open subset $V_1 \times V_2$ of ${\mathcal U}_{C_1}(r,d_1)\times{\mathcal U}_{C_2}(r,d_2)$ and a polarization $w$ such that $\phi|_{\mathscr{U}\cup\mathscr{V}}$ is a morphism,  where we set $\mathscr{V}=\pi^{-1}(V_1 \times V_2)$.  
}

Then, in analogy with the smooth case, for any $L\in \Pic(C)$ we define the variety ${\mathcal SU}_C(w,r,L)$ which is, roughly, the closure in $\U_C(w,r,\chi)_{d_1,d_2}$ of the locus parametrising classes of vector bundles with fixed determinant $L$ where $d_i=\deg(L|_{C_i})$. When $r$ and $d_i$ are coprime, as in the smooth case, we obtain the following result,  see Theorem \ref{THM2}:
\hfill\par
{\bf Theorem C}
{\it Under the hypothesis of Theorem A, ${\mathcal SU}_C(w,r,L)$ is a rational variety.}

Recent results concerning rationality of  these moduli spaces on reducible  curves are obtained in \cite{DS18} and \cite{BDS16} in the case of rank two and in \cite{BhB14} for an integral irreducible nodal curve. 
\hfill\par

The paper is organized as follows. In Section \ref{sec1} we fix notations about reducible nodal curves. In Section \ref{sec2} we introduce the notion of depth one sheaves, of polarization and $w$-semistability and we recall general properties on their moduli spaces. In Section \ref{sec3} we introduce the projective bundle ${\mathbb P}({\mathcal F})$, we define the sheaf $E_u$ associated to  $u \in {\mathbb P}({\mathcal F})$ and we study when it is $w$-semistable.  In Section \ref{sec4} we prove Theorems A and B. Finally, in Section \ref{sec5} we deal 
with moduli spaces with fixed determinant and we prove Theorem C. 

{\it Acknowledgements} 

\noindent We would like to thank Alessandro Verra for 
comments on a preliminary version of this paper and the  referee for several valuable advices. We are grateful to  Prof. P.E. Newstead and Prof. A. Dey for suggesting us some   references.

\section{Nodal reducible curves}
\label{sec1}
In this paper we will consider nodal reducible complex projective curves with  two smooth irreducible components and one single node. Let $C$ be  such a curve, we  consider   a normalization map   
  $\nu : C_1 \sqcup C_2 \to C$,
where  $C_i$  is a smooth irreducible curve of genus $g_i \geq 1$. Hence $\nu^{-1}(x)$ is a single point except when $x $ is the node $p$ of $C$, in this case   $\nu^{-1}(p)=\{q_1,q_2\}$ with $q_j\in C_j$.  Since the restriction ${\nu}_{\vert C_i}$ is an isomorphism  we will identify $C_1$  and $C_2$ with the irreducible components of $C$.  
\hfill\par

Notice that $C$ can be embedded in a smooth surface $X$,
on which  $C$ is an effective divisor $C = C_1 + C_2$  with $C_1C_2 = 1$.
Let $J_C = O_X(-C)$ and $J_{C_{i}} = O_X(-C_i)$ be  the ideal sheaves of $C$ and $C_i$ respectively  in $X$, then we have the inclusion $J_C \subset J_{C_i}$. 
We have  the following commutative diagram
$$\xymatrix{
& 0\ar[r] & \OO_X(-C)\ar[r]\ar@{^{(}->}[d] & \OO_X(-C_2)\ar[r]\ar@{^{(}->}[d] & \OO_{C_1}(-C_2)\ar[r]  & 0 \\
 & & \OO_X\ar[r]^{\simeq }\ar@{->>}[d] & \OO_X\ar@{->>}[d] \\
 0\ar[r] & J_{C_2}/J_{C} \ar[r] & \OO_C\ar[r] & \OO_{C_2}\ar[r] & 0
}
$$
from which one deduces the isomorphism $J_{C_2}/J_{C} \simeq O_{C_1} (-C_2) $. This gives the exact sequence
\begin{equation}
\label{decomposition}
 0 \to O_{C_1} (-C_2)  \to O_C \to O_{C_2} \to 0.
\end{equation}
which is called the \textit{decomposition sequence of $C$}. From it we can compute the Euler characteristic of $O_C$:
 $$\chi(O_C)= \chi(O_{C_1}(-C_2)) + \chi(O_{C_2}).$$
Let $p_a(C) = 1 - \chi(O_C)$ be the  {\it arithmetic genus} of $C$, from the above relation we get that  
 $ p_a(C) = g_1+ g_2.$

\begin{notation}
\label{notazioni}
We will denote by $j_i \colon C_i \hookrightarrow C$ the natural inclusion of $C_i$ in $C$. We will denote by $\OO_{q_i}$ the stalk of $(j_i)_*\OO_{C_i}$ in $p$ and by $\OO_p$   the stalk of $\OO_C$ in $p$.
\end{notation}


\section{Moduli space of depth one  sheaves}
\label{sec2}
Let $C$ be a smooth irreducible  projective curve of genus $g \geq 1$.  The moduli space of semistable vector bundles of rank $r$ and degree $d$ on $C$ will be denoted by ${\mathcal U}_{C}(r,d)$. Its points are  $S$-equivalence classes of  semistable vector bundles  on the curve. We will denote by $[E]$ the class of  a vector bundle $E$. 
In \cite{S} it is proved that  ${\mathcal U}_{C}(r,d)$ is an irreducible and projective variety. Moreover, see \cite{S}, \cite{Tu}, we have:
\begin{equation}
\dim {\mathcal U}_{C}(r,d) = \begin{cases} r^2(g-1) +1 & g \geq 2 \\
\gcd(r,d) & g=1. 
\end{cases}
\end{equation}

In particular, when $r$ and $d$ are coprime, ${\mathcal U}_{C}(r,d)$ is a smooth variety, whose points parametrizes isomorphism classes of stable vector bundles. Moreover, for $g = 1$, we also have an isomorphism ${\mathcal U}_C(r,d)  \simeq C$ (see \cite{At} and \cite{Tu}). 
 \hfill\par
Let $C$ be a  nodal curve with a single node $p$ and two smooth irreducible components $C_1$ and $C_2$.  
To construct compactifications of moduli spaces of vector bundles on $C$  we introduce depth one sheaves, following  the approach of Seshadri, see  \cite{S}.

\begin{definition} A coherent sheaf $E$ on $C$  is of {\it depth one}  if every torsion section  vanishes identically on some components of $C$.
\end{definition}

A coherent sheaf $E$ on $C$ is of depth one if and only if the stalk at  the node $p$ is isomorphic to  $\OO_p^{a}\oplus \OO_{q_1}^{b}\oplus\OO_{q_2}^{c}$, see  \cite{S}. In particular, any vector bundle $E$ on $C$ is a sheaf of depth one.  If $E$ is a sheaf of depth one on $C$, then its restriction $E\vert_{C_i}$ is a  torsion free sheaf on $C_i \setminus p$ (possibly identically zero). Moreover, any subsheaf of $E$ is of depth one too. 
\hfill\par
Let $E$ be a sheaf of depth one on $C$. We define the {\it relative rank} of $E$ on the component $C_i$ as the rank of the restriction $E_i = E_{\vert C_i}$ of $E$ to  $C_i$
\begin{equation} r_i= \Rk(E_i)
\end{equation}
and the {\it multirank} of $E$ as the pair $(r_1,r_2)$.
We define the {\it relative degree} of $E$ with respect to the component $C_i$ as the degree of the restriction  $E_i$
\begin{equation}
d_i = deg(E_i) = \chi(E_i) - r_i\chi(O_{C_i}),
\end{equation}
where $\chi(E_i)$ is the Euler characteristic of $E_i$. The {\it multidegree} of $E$ is the pair $(d_1,d_2)$.

\begin{definition}
A {\it polarization} $w$ of $C$ is given by a pair of rational weights $(w_1,w_2)$ such that $0 < w_i < 1$ and $w_1 + w_2 = 1$. For any sheaf $E$ of depth one on $C$, of multirank $(r_1,r_2)$ and $\chi(E) = \chi$, we define  the {\it polarized slope} as
$$ \mu_w(E)= \frac{\chi}{w_1r_1+w_2r_2}.$$
\end{definition}

\begin{definition}
Let $E$ be a sheaf of depth one on $C$. $E$ is said $w$-{\it semistable} if for any subsheaf $F \subseteq E$  we have $\mu_w(F) \leq \mu_w(E)$;  $E$ is said $w$-{\it stable} if $\mu_w(F) < \mu_w(E)$ for all proper subsheaf $F$ of $E$.
\end{definition}

For each  $w$-semistable sheaf $E$ of depth  one on $C$ 
 there exists a finite filtration of sheaves of depth one on $C$:
$$ 0= E^0 \subset E^1 \subset E^2 \subset \dots \subset E^k = E,$$
such that each quotient $E^i / E^{i-1}$ is a $w$-stable sheaf of depth one on $C$ with polarized slope $\mu_w(E^i / E^{i-1}) = \mu_w(E)$.
This is called  a {\it Jordan-Holder filtration} of $E$. 
The sheaf 
$$Gr_w(E) =\oplus_{i=1}^kE^i /E^{i-1}$$
is said the {\it graduate sheaf associated to} $E$ 
and it depends only on the isomorphism class of $E$.
Let $E$ and $F$ be $w$-semistable  sheaves of depth one on $C$. We say that  $E$ and $F$ are $S_w$-equivalent if and only if 
$Gr_w(E) \simeq  Gr_w(F)$.  If $E$ and $F$ are $w$-stable sheaves  then $S_w$-equivalence is just isomorphism, as  in  the smooth case.
\hfill\par
There exists a moduli space ${\mathcal U}^s_C(w,(r_1,r_2),\chi)$   parametrizing isomorphism classes of  $w$-stable  sheaves of depth one 
 on $C$ of multirank $(r_1,r_2)$ and given Euler characteristic $\chi$, see \cite{S}.
It has  a  natural compactification     ${\mathcal U}_C(w,(r_1,r_2),\chi)$,  whose points   correspond to $S_w$-equivalence classes   
of $w$-semistable    sheaves of depth one on $C$   of multirank $(r_1,r_2)$ and given Euler characteristic $\chi$. In particular, when $r_1= r_2 = r$,  we denote by   ${\mathcal U}_C(w,r,\chi)$ the corresponding moduli space. In this case we have the following result (see \cite{T1} and \cite{T2}):
\begin{theorem}
\label{modulispace}
Let $C$ be a nodal curve  with  a single node $p$ and two smooth  irreducible components $C_i$ of genus $g_i \geq 1$, $i = 1,2$.
 For a generic polarization $w$  we have the following properties:
\hfill\par\noindent
\begin{enumerate}
\item{} any $w$-stable vector bundle $E \in {\mathcal U}_C(w,r,\chi)$  satisfies the following condition:
\begin{equation}
 \label{eq:irreduciblecomponent}   
 w_i \chi(E) \leq \chi (E_i) \leq w_i \chi(E) + r,
\end{equation}
\noindent
where $E_i$ is the restriction of  $E$ to $C_i$;
\item{}  if a vector bundle $E$ on $C$ satisfies  the above condition  for $i = 1,2$  and the restrictions  $E_1$ and $E_2$ are semistable vector bundles, then $E$ is $w$-semistable. 
Moreover, if at least one of the restrictions is stable, then $E$ is $w$-stable;
\item{} the moduli space ${\mathcal U}_C(w,r,\chi)$ is connected, each irreducible component has dimension
$ r^2(p_a(C) -1) +1$ and it corresponds to the choice of a multidegree $(d_1,d_2)$ satisfying conditions 
\ref{eq:irreduciblecomponent}.
\end{enumerate}
\end{theorem}

\begin{definition} We denote by ${\mathcal U}_C(w,r,\chi)_{d_1,d_2}$
 the irreducible component of   ${\mathcal U}_C(w,r,\chi)$ corresponding to the multidegree $(d_1,d_2)$.
\end{definition}

\section{Construction of  depth one sheaves.}  
\label{sec3}

In this section  we deal with  construction of depth one sheaves on a  nodal curve $C$ with two irreducible components and a single node. We begin with the following lemma: 
\begin{lemma}
\label{LEM:projbundle}
Let $C_1$ and $C_2$ be smooth complex projective curves of genus  $g_i\geq 1$, $i = 1,2$, and $q_i \in C_i$. Fix $r\geq 2$ and $d_1,d_2\in \Z$ such that $r$ is coprime with both $d_1$ and $d_2$. Then, there exists a projective bundle $$\pi:\PP(\mathcal{F})\rightarrow \mathcal{U}_{C_1}(r,d_1)\times \mathcal{U}_{C_2}(r,d_2)$$
such that the fiber over $([E_1],[E_2])$ is
$\PP(\Hom(E_{1,q_1},E_{2,q_2}))$, where $E_{i,q_i}$ is the fiber  of $E_i$ at the point $q_i$. 
\end{lemma}

\begin{proof}
We recall that, as $r$ and $d_i$ are coprime, there exists a  Poincar\'e bundle $\mathcal{P}_i$ for the moduli space of semistable vector bundles on $C_i$ of rank $r$ and degree $d_i$,
i.e. a vector bundle $\mathcal{P}_i$ on $\mathcal{U}_{C_i}(r,d_i) \times C_i$ such that $\mathcal{P}_i\vert_{[E_i] \times C_i} \simeq E_i$, under the identification $[E_i]\times C_i\simeq C_i$. This follows from a result of  \cite{R} if $g_i \geq 2$  and from the isomorphism ${\mathcal U}_{C_i}(r,d_i) \simeq C_i$  when  $g_i= 1$. 
\hfill\par
For $i= 1,2$, consider the natural  inclusion
$$\iota_i:\mathcal{U}_{C_i}(r,d_i) \times q_i \hookrightarrow  \mathcal{U}_{C_i}(r,d_i) \times C_i,$$
and the pull back ${{\iota}_i}^*({\mathcal P}_i)$ of the Poincar\'e bundle.
Since $\mathcal{U}_{C_i}(r,d_i) \times q_i $ is isomorphic to $\mathcal{U}_{C_i}(r,d_i)$, ${{\iota}_i}^*({\mathcal P}_i)$ can be seen as a vector bundle on $\mathcal{U}_{C_i}(r,d_i)$ of rank $r$ whose fiber at $[E_i]$ is actually $E_{i,q_i}$. 
\hfill\par
Note that the product $\mathcal{U}_{C_1}(r,d_1)\times \mathcal{U}_{C_2}(r,d_2)$ 
is a smooth irreducible variety. 
Let ${p}_1$ and $p_{2}$ denote the projections of the product  onto  factors.  We define on $\mathcal{U}_{C_1}(r,d_1)\times \mathcal{U}_{C_2}(r,d_2)$ the following sheaf: 
\begin{equation}
\label{F}
     \mathcal F \colon = {\mathcal Hom}(p_1^*({{\iota}_1}^*({\mathcal P}_1)), p_2^*({{\iota}_2}^*({\mathcal P}_2))).
 \end{equation}
By construction, $\mathcal{F}$ is a vector bundle of rank $r^2$ whose fiber at the point $([E_1],[E_2])$ is $\Hom(E_{1,q_1},E_{2,q_2})$. By taking the associated projective bundle we conclude the proof.
\end{proof}

\noindent Let $C_1$ and $C_2$ be smooth irreducible curves, we consider a nodal curve $C$ with two smooth components and  a single node $p$ which is obtained by identifying the points $q_1\in C_1$ and $q_2\in C_2$. Let  $E_i$ be a stable vector bundle of rank $r$  and  degree $d_i$ on $C_i$ and consider 
a non zero homomorphism $\sigma:E_{1,q_1}\to E_{2,q_2}$  between the fibres. Assume that the rank of $\sigma$ is $k$,  with $ 1 \leq k \leq r$. We can associate to these  data a depth one sheaf on the nodal curve $C$, roughly speaking, by gluing the vector bundles $E_1$ and $E_2$ along the fibers (at $q_1$ and $q_2$ respectively) with the homomorphism $\sigma$, as follows. 

\noindent 
Let $j_p$ be the inclusion of $p$ in $C$ and let
$j_i \colon C_i \to C$ be the inclusion of $C_i$ in $C$, for $i = 1,2$.  
The sheaf ${j_i}_*E_i$ is a depth one sheaf on $C$ whose stalk at $p$ is the stalk of $E_i$ at  $q_i$. Hence, there is a natural surjective map given by restriction onto the fiber of $E_i$ at $q_i$, i.e. the map
$$\rho_i \colon {j_i}_*E_i \to E_{i,q_i}.$$
The sheaf   ${j_1}_* (E_1) \oplus {j_2}_*(E_2)$ is of depth one on $C$  and we have a  surjective map  
$$\rho_1 \oplus \rho_2 \colon {j_1}_*E_1 \oplus {j_2}_* E_2 \to E_{1,q_1}\oplus E_{2,q_2}.$$
The sheaf  ${j_p}_*{j_p}^*{j_2}_*(E_2)$ has depth one too, and it is a skyscraper sheaf over $p$ whose stalk is $E_{2,q_2}$. 
So we have again a surjective map 
$$\rho :{j_p}_*{j_p}^*{j_2}_*(E_2) \to E_{2,q_2}.$$
Let $\sigma \colon E_{1,q_1} \to E_{2,q_2}$ be a non zero homomorphism and consider the induced surjective map 
$$ \sigma \oplus id \colon E_{1,q_1} \oplus E_{2,q_2} \to Im(\sigma) \oplus E_{2,q_2}.$$
We have, moreover, the map 
$$\delta:Im(\sigma)\oplus E_{2,q_2}\to E_{2,q_2}$$
which sends $(u,v)$ to $u-v$. We denote by $\Delta \subset  Im(\sigma)\oplus Im(\sigma)$ the diagonal. By construction we have $\Delta \simeq {\mathbb C}_p^k$. 

Finally we define the map  of sheaves
$$\tilde{\sigma}:{j_1}_*(E_1) \oplus {j_2}_*(E_2) \to {j_p}_*{j_p}^*{j_2}_*(E_2)$$ by requiring that the following diagram commutes.

\begin{equation}
\label{EQ:DEFsigmatilde}
\xymatrix@R=1pc{
&
    K_1\oplus K_2 \ar@{=}[rr]\ar@{^{(}->}[d] & 
    &
    K_1\oplus K_2 \ar@{^{(}->}[d] \\
0 \ar[r] &
    \ker \tilde{\sigma} \ar[rr]\ar@{->>}[dd] &
    &
    {j_1}_*(E_1) \oplus {j_2}_*(E_2) \ar@{->>}[dd]\ar[r]^-{\tilde{\sigma}}
        \ar@/_1pc/[ld]^-{\rho_1\oplus\rho_2}&
    {j_p}_*{j_p}^*{j_2}_*(E_2)\ar@{->>}[dd]^-{\rho}\ar[r] &
    0 \\
&
    &
    E_{1,q_1}\oplus E_{2,q_2}\ar@/_1pc/[rd]^-{\sigma\oplus\id} \\
0 \ar[r] & 
    \Delta \ar[rr] &
    &
    Im(\sigma)\oplus E_{2,q_2} \ar[r]_-{\delta} &
    E_{2,q_2}\ar[r] &
    0
}\end{equation}

It follows immediately by construction  that $\ker \tilde{\sigma}$ is a sheaf of depth one on $C$, which coincides with $E_i$ on $C_i \setminus p$. One can easily see that the isomorphism class of $\ker \tilde{\sigma}$ does not depend on the isomorphism classes of $E_i$. Moreover, the same happens if one uses $\sigma'=\lambda\sigma$ with $\lambda\in \C^*$, instead of $\sigma$. 
\vspace{2mm}

\noindent From now on, we will assume to be under the hypotesis of Lemma \ref{LEM:projbundle}. Let $\PP(\mathcal{F})$ be the projective bundle on ${\mathcal U}_{C_1}(r,d_1) \times {\mathcal U}_{C_2}(r,d_2)$.
We can conclude that the construction of $\ker \tilde{\sigma}$ depends on the data contained in $u = (([E_1],[E_2]), [\sigma])\in \PP(\mathcal{F})$ and not on the particular choices of $E_1,E_2$ and $\sigma$.

\begin{definition}
We will denote  by $E_u$ the  kernel of $\tilde{\sigma}$ defined by $u \in {\mathbb P}({\mathcal F})$.  \end{definition}

The above construction gives the following: 

\begin{proposition}
\label{construction}
Let $E_u$  be the sheaf defined by $u = (([E_1],[E_2]), [\sigma]) \in {\mathbb P}({\mathcal F})$.  Then $E_u$ is a depth one sheaf on $C$ with $\chi(E_u)= \chi(E_1) + \chi(E_2) -r$ and  multirank $(r,r)$. It is a vector bundle if and only if $\sigma $ is an isomorphism. In this case,  ${E_u}_{\vert C_i} = E_i$. 
\end{proposition}
\begin{proof}
Let $\Rk (\sigma ) = k$.  Since $E_u$ is a depth one sheaf,  the stalk of $E_u$ at the node $p$ is isomorphic to
$$ \OO_{p}^{a}\oplus\OO_{q_1}^{b}\oplus\OO_{q_1}^{c},$$ where $a+b=\Rk(E_u|C_1)= r$ and $a+c=\Rk(E_u|C_2) = r$  (see Section \ref{sec2}). From the diagram \ref{EQ:DEFsigmatilde}, it follows that the rank of the free part of the stalk of $E_u$ in $p$ is $k$, so $a=k$. Hence we have
$E_u|_p \simeq \OO_p^{k}\oplus \OO_{q_1}^{r-k}\oplus\OO_{q_2}^{r-k}$. In particular, $E_u$ is a vector bundle if and only if $k=r$, i.e., exactly when $\sigma$ is an isomorphism.
\end{proof}

In order to obtain a $w$-semistable sheaf, for some polarization $w$,  we have the following necessary condition:
\begin{lemma}
\label{necessary condition}
Let $E=E_u$ be the sheaf defined by $u = (([E_1],[E_2]), [\sigma]) \in {\mathbb P}({\mathcal F})$ and let $k$ be the rank of $\sigma$. Then, if $E$ is $w$-semistable for some $w$, the following conditions are satisfied:
 \begin{equation}
 \label{eqnecessary}
 \begin{cases}
 \chi(E) w_1 \leq \chi(E_1) \leq \chi(E) w_1 + k,
 \\
 \chi(E) w_2 + r-k \leq \chi(E_2) \chi(E) w_2 + r.
 \end{cases}
 \end{equation}
\end{lemma}
\begin{proof}
Assume that $E$ is $w$-semistable for a polarization $w$. Let  $K_1$ be  the kernel of the map 
$$\sigma \circ \rho_1 \colon {j_1}_*E_1  \to Im \sigma,$$
and $K_2$ be  the kernel of the map
 $ \rho_2 \colon {j_2}_*E_2  \to E_{2,q_2}$, 
as in diagram \ref{EQ:DEFsigmatilde}. Since $K_i$ is a subsheaf of $E$, then by $w$-semistability of $E$ we   must have $ \mu_w(K_i) \leq \mu_w(E)$.
We have:    $\mu_w(K_1) = \frac{\chi(K_1)}{w_1r} = \frac{\chi(E_1)-k}{w_1r} \leq \frac{\chi(E)}{r}$, 
which implies 
$$ \chi(E_1)  \leq \chi(E)w_1 + k.$$
By replacing $\chi(E_1) = \chi(E) - \chi(E_2) +r$ in the above inequality, we obtain:
$$\chi(E_2) \geq \chi(E) w_2 + r-k.$$
Finally, we have $\mu_w(K_2) = \frac{\chi(K_2)}{w_2r} = \frac{\chi(E_2)-r}{w_2r} \leq \frac{\chi(E)}{r}$, 
which implies 
$$ \chi(E_2)  \leq \chi(E)w_2 + r.$$
Again, by replacing $\chi(E_2)= \chi(E) - \chi(E_1) +r$ we obtain  $\chi(E_1) \geq \chi(E) w_1$. 
\end{proof}

Given $u = (([E_1],[E_2]), [\sigma])$ and $E_u$  defined by $u$, we wonder if there exists a polarization $w$ such that the above conditions \ref{necessary condition} hold. 
The answer depends only on numerical assumptions on
$(\chi(E_1), \chi(E_2))$ and $\Rk \sigma$ as it is shown in the following lemma.

\begin{lemma} 
\label{existencew}
Let $r \geq 2$ and $ 1 \leq k  \leq r$ be integers. 
There exists a non empty subset ${\mathcal W}_{r,k} \subset {\mathbb Z}^2$ such that for any pair $(\chi_1,\chi_2) \in {\mathcal W}_{r,k}$  we can find a polarization $w$ satisfying the following conditions:
\begin{equation}
\label{numconditions}
\begin{cases}\chi w_1  \leq \chi_1 \leq \chi w_1 + k, \\
\chi w_2 + r-k \leq \chi_2 \leq \chi w_2 + r,
\end{cases}
\end{equation}
where $\chi = \chi_1 + \chi_2 -r$.
\end{lemma}
\begin{proof}
First of all note that if $\chi = 0$, i.e.  
$\chi_1+\chi_2= r$ and we assume that $ 0 \leq \chi_1 \leq r$, then any polarization $w$ satisfies conditions \ref{numconditions}. 
\hfill \par
We distinguish two cases according to the sign of $\chi$. Assume that $\chi > 0$. 
Then there exists a polarization $w$ satisfying  conditions \ref{numconditions}, 
if and only if  the following system has solutions:  
$$  \begin{cases} \frac{\chi_1 - k}{\chi} \leq w_1 \leq \frac{\chi_1}{\chi} 
\\ 
\frac{\chi_2 - r}{\chi} \leq w_2 \leq \frac{\chi_2 + k -r }{\chi} 
\\
w_1 + w_2 = 1
\\ 0 < w_i < 1, w_i \in {\mathbb Q}
\end{cases}
$$
This occurs  if  and only if $\chi_1 > 0$ and $\chi_2 > r-k$.
Likewise,  if $\chi <0$,
then we have the following system:
$$  \begin{cases} \frac{\chi_1 }{\chi} \leq w_1 \leq \frac{\chi_1 - k}{\chi} 
\\ 
\frac{\chi_2 - r +k}{\chi} \leq w_2 \leq \frac{\chi_2  -r }{\chi} 
\\
w_1 + w_2 = 1
\\ 0 < w_i < 1, w_i \in {\mathbb Q}
\end{cases}
$$
which  has solutions if and only if $\chi_1 <k$ and $\chi_2 <r$.
\end{proof}

\begin{center}
\includegraphics[width=0.7\textwidth]{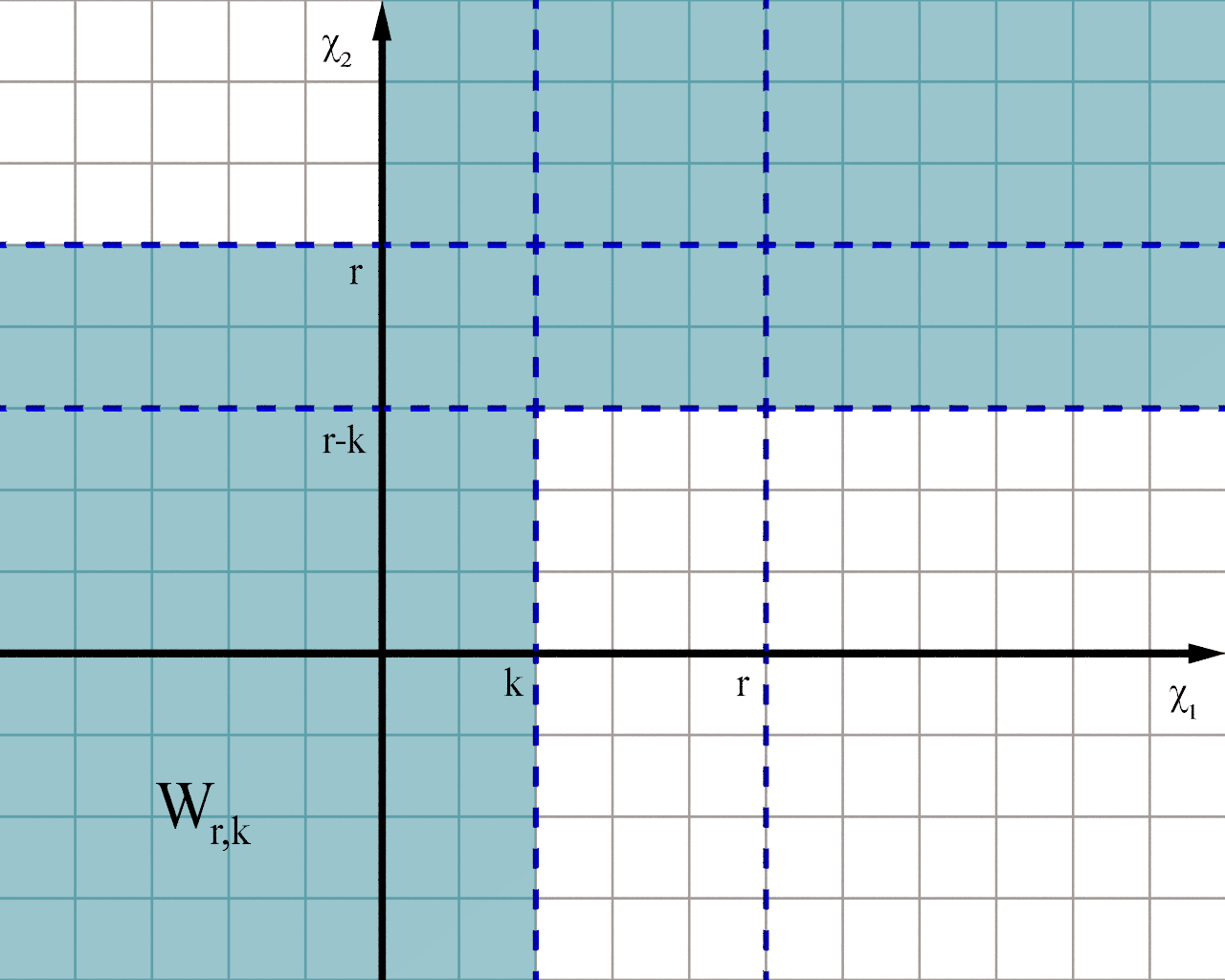}
\end{center}

\begin{remark}
\label{allk}
Let ${\mathcal W}_r= \bigcap_{k = 1}^{r} {\mathcal W}_{r,k}$.  Note that it is a  non empty subset and it is actually ${\mathcal W}_{r,1}$. Moreover,
if  $(\chi_1,\chi_2) \in {\mathcal W}_r$, then by the proof of lemma \ref{existencew} it follows  that we can find a polarization $w$ which satisfies the conditions \ref{numconditions} for all $k= 1, \dots,r.$ 
\end{remark}

Assume that $\Rk \sigma = r$, i.e. $E$ is a vector bundle, then  the necessary conditions of Lemma \ref{necessary condition} are the same  of   Theorem
\ref{modulispace}. Hence, by the above Theorem, 
they are also sufficient to give $w$-semistability of $E$. So we obtain the following: 

\begin{corollary}
\label{COR:NECISSUF}
Let $E= E_u$ be the sheaf defined by $u = (([E_1],[E_2]), [\sigma]) \in {\mathbb P}({\mathcal F})$. Assume that $\Rk \sigma = r$ and  $(\chi(E_1), \chi(E_2)) \in {\mathcal W}_{r,r}$, then there exists  a polarization $w$ such that   $E$ is $w$-semistable. 
In particular, since $E_i$ are stable, then $E$ is $w$-stable too. 
\end{corollary}

Unfortunately, when $E_u$ fails to be  a vector bundle, the  necessary conditions of Lemma 
\ref{necessary condition} are not enough to  ensure $w$-semistability, see \cite{T2} for an example. Nevertheless,  we are able to produce an open subset 
of  $\mathcal{U}_{C_1}(r,d_1)\times \mathcal{U}_{C_1}(r,d_1)$ such that for every $u$ over this open, the sheaf $E_u$ is $w-$semistable.

\hfill\par
We recall the following definition, see \cite{NR}.
\begin{definition} Let $G$ be a vector bundle  on a smooth curve. For any  integer $k$ we set:
$$\mu_k(G)= \frac{ deg (G) +k}{rk(G)}.$$
 A vector bundle $G$ is said  $(m,k)$-semistable (resp. stable) if for any subsheaf $F$ we have:
$$ \mu_{m}(F) \leq \mu_{m-k}(G)  \quad (resp. <).$$
\end{definition}

\begin{proposition}
\label{PROP:STABILITY}
Let $E = E_u$  be the sheaf defined by $u = (([E_1],[E_2]), [\sigma]) \in {\mathbb P}({\mathcal F})$. Assume that   $\Rk \sigma = k \leq r-1$. 
If $(\chi(E_1), \chi(E_2)) \in {\mathcal W}_{r,k}$,   $E_1$ is $(0,k)$-semistable and $E_2$ is  $(0,r)$-semistable, then there exists a polarization $w$ such that $E$ is $w$-semistable. Moreover, if $E_1$ is $(0,k)$-stable or $E_2$ is $(0,r)$- stable then $E$ is $w$-stable too. 
\end{proposition}
\begin{proof}
Since  $(\chi(E_1), \chi(E_2)) \in {\mathcal W}_{r,k}$, by Lemma \ref{existencew} there exists a polarization $w$ such that the necessary conditions \ref{necessary condition} hold. 
We claim that if $E_1$ is $(0,k)$-semistable and $E_2$ is $(0,r)$-semistable, then $E$ is $w$-semistable.
\hfill\par\noindent
Let $F \subset E$ be a subsheaf, it is a sheaf of depth one too. Assume that $F$ has multirank $(s_1, s_2)$ and that  at the node $p$ the stalk  of $F$ is $\OO_p^s \oplus \OO_{q_1}^a \oplus \OO_{q_2}^b$, with $s \geq 0$, 
$s_1 = s+a \leq r$ and $s_2 = s + b \leq r$. 
 Since $\Rk \sigma = k$, by construction the free part of the stalk of $E$  at $p$ is  $\OO_p^k$.
This implies that  $0 \leq s \leq k$. 
\hfill\par
By construction, there exists two vector bundles $F_1 \subseteq E_1$ and $F_2 \subseteq E_2$ such that $F$ is the kernel of the  restriction of $\tilde{\sigma}$ to the subsheaf 
$ {j_1}_*(F_1) \oplus {j_2}_*(F_2)$:
$$ {\tilde{\sigma}}_{\vert {j_1}_*(F_1) \oplus {j_2}_*(F_2)} \colon {j_1}_*(F_1) \oplus {j_2}_*(F_2) \to
{j_p}_*{j_p}^*{j_2}_*(E_2).$$
By proceding as in the diagram \ref{EQ:DEFsigmatilde}, we deduce that $F$ fits into an exact sequence as follows:
$$ 0 \to G_1 \oplus G_2 \to F \to {\mathbb C}_p^s \to 0,$$
where $G_1$ is the kernel of $(\sigma \circ \rho_1)|_{F_1}$ and $G_2$ is the kernel of 
$\rho_2|_{F_2}$. Hence $G_i \subseteq K_i$. 
Note that if $s= 0$, then actually $F \simeq G_1 \oplus G_2$.
\hfill\par
For any $s$, we compute the $w$-slope of $F$: 
$$\mu_w(F) = \frac{\chi(F)}{w_1s_1 + w_2 s_2}=
\frac{\chi(G_1) + \chi(G_2) +s}{w_1s_1 + w_2 s_2}=$$
$$ = \frac{deg(G_1) + s_1(1-g_1) + deg(G_2) + s_2(1-g_2) + s}{w_1s_1 + w_2 s_2}.$$
Since  $E_1$ is $(0,k)$-semistable,  then we have: 
$$\frac{deg (G_1)}{s_1} \leq  \frac{d_1 -k}{r},$$
since $E_2$ is $(0,r)$-semistable, then $E_2(-q_2)$ is $(0,r)$-semistable too,  so we have:
$$\frac{deg (G_2)}{s_2} \leq  \frac{d_2 -2r}{r}.$$
By replacing we obtain:
\begin{multline}
\mu_w(F) \leq \frac{1}{w_1s_1 + w_2 s_2}\left[
s_1w_1\left(\frac{(d_1-k) +r(1-g_1)}{w_1r}\right) + s_2w_2\left(\frac{(d_2-r) + r(1-g_2)}{w_2r}\right) 
+s-s_2\right]=\\= \frac{s_1 w_1}{w_1s_1 + w_2 s_2} \mu_w(K_1) +
\frac{s_2 w_2}{w_1s_1 + w_2 s_2} \mu_w(K_2) +
\frac{s-s_2}{w_1s_1 + w_2s_2}.
\end{multline}
By Lemma \ref{necessary condition}, we have that $\mu_w(K_i) \leq \mu_w(E)$, so we obtain:
$$\mu_w(F) \leq  \mu_w(E)+ \frac{s -s_2}{w_1s_1 + w_2s_2}.$$
Since $s-s_2 \leq0$,  we have that $\mu_w(F) \leq  \mu_w(E)$. 
\hfill\par
Finally, if $E_1$ is $(0,k)$-stable or $E_2$ is $(0,r)$-stable, then the above inequality is strict. This concludes the proof. 
\end{proof}

Note that, by definition, if $E_i$ is $(0,r)$-stable, then it is also $(0,k)$-stable for all $k\leq r$. 

\begin{lemma}
\label{LEM:OPENSTABLE}
Let ${\mathcal U}_{C_i}(r,d_i)$ be the moduli space of semistable vector bundles of rank $r$ and degree $d_i$ on a 
smooth curve $C_i$ of genus $g_i$. 
If $d_i$ and $r$ are coprime and   $g_i > r +1$, then the locus of vector bundles of ${\mathcal U}_{C_i}(r,d_i)$ which are $(0,r)$-stable is a non empty open subset of ${\mathcal U}_{C_i}(r,d_i)$.
\end{lemma}
\begin{proof}
Let us consider the locus
$$Y = \{ [E] \in {\mathcal U}_{C_i}(r,d_i)\ \vert \  \text{$E$ is not}\  (0,r)-\mbox{stable} \}.$$
We can consider the subset  $Y_{a,s}$ of $Y$ given by all stable  vector bundles  $E$ which can be written as
$$0 \to F \to E \to Q \to 0,$$
where $F$ is a subbundle of $E$ with $\deg(F)= a$ and $\Rk(F)=s \leq r-1$ and  $$\mu(E)-1 =\mu_{-r}(E)\leq \mu(F) \leq \mu_0(E)=\mu(E).$$
By a deformation argument (see the proof of Proposition 1.4 of \cite{TR}), one can prove that  if $Y_{a,s} \not= \emptyset$ then  for a general $E$ in $Y_{a,s}$ both $F$ and $Q$ are stable. Moreover,  since $E$ is stable, we have $\Hom(Q,F)=0$. Hence we can write
$$ \dim Y_{a,s} \leq \dim {\mathcal U}_{C_i}(s,a) + 
\dim {\mathcal U}_{C_i}(r-s,d_i-a) + 
\dim H^1(C_i, {\mathcal Hom}(Q,F)) -1 = $$
$$ = (g_i-1)(r^2 -rs + s^2) +1 + (d_is- ar).$$
Hence: 
$$\dim {\mathcal U}_{C_i}(r,d_i) - \dim Y_{a,s} \geq (g_i-1)(rs-s^2)- (d_is-ar).$$
Since $E \in Y$, $\mu_0(F) \geq \mu_{-r}(E)$, i.e.
$$\frac{a}{s} \geq \frac{d_i -r}{r},$$
which implies 
$$ d_is-ar \leq rs.$$
Finally, if $g_i > 1+r$, for all $s\leq r-1$, we have:
$$\dim {\mathcal U}_{C_i}(r,d_i) - \dim Y_{a,s} \geq
s[(g_i-1)(r-s)-r] > 0$$
which concludes the proof. 
\end{proof}

\section{Main results}
\label{sec4}
In this section we prove our main results. We assume that the hypothesis of Lemma \ref{LEM:projbundle} are satisfied.
\hfill\par

Let ${\mathbb P}({\mathcal F})$ be the projective bundle on
${\mathcal U}_{C_1}(r,d_1) \times {\mathcal U}_{C_2}(r,d_2)$. 
For each $ 1 \leq k\leq r-1$, let $\mathcal{B}_k$ be the subset of $\PP(\mathcal{F})$ such that 
$$\mathcal{B}_k\cap \pi^{-1}([E_1],[E_2])=\{[\sigma]\in \PP(\Hom(E_{1,q_1},E_{2,q_2}))\, |\, \Rk(\sigma)\leq k\}.$$
It is a proper closed subvariety of $\PP(\mathcal{F})$. 

\begin{definition}
\label{openU}
We will denote by $\mathscr{U}$ the open subset given by the complement of  $\mathcal{B}_{r-1}$ in $\PP(\mathcal{F})$. 
\end{definition}

\begin{remark}
\label{REM:dim}
Note that $\dim \mathscr{U} = \dim \PP(\mathcal{F}) = r^2(g_1+g_2-1)+1$.
If we denote by $\pi_{\mathscr{U}}$ the restriction of $\pi$ to $\mathscr{U}$ we have that, by construction,
$$ \pi_{\mathscr{U}}:\mathscr{U}\rightarrow \mathcal{U}_{C_1}(r,d_1)\times \mathcal{U}_{C_2}(r,d_2)$$
is a fiber bundle whose fibers are isomorphic to $\PGL(r)$. More precisely
$$\pi_{\mathscr{U}}^{-1}([E_1],[E_2])=\PP(\GL(E_{1,q_1},E_{2,q_2})).$$
\end{remark}

For  $\chi= d_1 + d_2 + r(1 - g_1 -g_2)$,   let  $\mathcal{U}_C(w,r,\chi)_{d_1,d_2}$ be the irreducible component  of the moduli space of 
depth one sheaves on $C$ of rank $r$ and characteristic $\chi$ corresponding to the multidegree $(d_1,d_2)$, see Section \ref{sec2}.  
Let ${\mathcal V}_C(w,r,\chi)_{d_1,d_2} \subset {\mathcal U}_C(w,r,\chi)_{d_1,d_2}$ be the subset parametrizing classes of vector bundles.

\begin{theorem}
\label{THM:MAIN}
Let $C$ be  a nodal curve with a single node  $p$ and two smooth irreducible components $C_i$ of genus $g_i \geq 1$. Fix $r\geq 2$, for any $d_i \in {\mathbb Z}$  we set $\chi_i = d_i + r(1-g_i)$ and $\chi= d_1 + d_2 +r(1-g_1-g_2)$. Assume that $r$ is coprime with both $d_1$ and $d_2$ and  
$(\chi_1, \chi_2) \in {\mathcal W}_{r,r}$. Then 
there exists a polarization $w$ such that the map  
$$ 
\xymatrix{
\phi \colon {\mathbb P}({\mathcal F}) \ar@{-->}[r] & {\mathcal U}_C(w,r,\chi)_{d_1,d_2}
}
$$
sending $u \to [E_u]$ is birational. 
 In particular, the restriction ${\phi}\vert_{\mathscr{U}}$ is a an injective morphism and the image $\Phi(\mathscr{U})$ is contained in   ${\mathcal V}_C(w,r,\chi)_{d_1,d_2} $. 
\end{theorem}
\begin{proof}
Let  $u = (([E_1],[E_2]),[\sigma]) \in \mathbb P({\mathcal F})$ and consider the sheaf $E = E_u$ defined by $u$, as in Section \ref{sec3}. 
Since  $(\chi_1,\chi_2) \in {\mathcal W}_{r,r}$, then, 
as a consequence of  Lemma \ref{existencew} and Corollary \ref{COR:NECISSUF} there exists a polarization $w$ such that $E_u$ is $w$-semistable for every $u \in \mathscr{U}$. This gives a point in  the moduli space $\mathcal{U}_C(w,r,\chi)_{d_1,d_2}$ and it shows that $\phi$ is well defined at least on $\mathscr{U}$.
\hfill\par

Now we will prove that ${\phi}_{\vert {\mathscr U}}$ is injective. Let $u=(([E_1],[E_2]),[\sigma])$ and $u'=(([E_1'],[E_2']),[\sigma'])$ in $\mathscr{U}$ with $\phi(u)=[E]$ and $\phi(u')=[E']$. Assume that $\phi(u)=\phi(u')$. Since $E$ and $E'$ are both $w$-stable and are in the the same $S_w$-equivalence class, then they have to be isomorphic (see Section \ref{sec2}). Let $\tau: E \to E'$ be an isomorphism. This induces an isomorphism $\tau_i:E_i\to E_i'$. So we can assume, that $E_i'=E_i$ and, thus $\sigma,\sigma':E_{1,q_1}\rightarrow E_{2,q_2}$ and $\tau_i:E_i\to E_i$ are isomorphism. As $E_p$ (respectively $E_p'$) is obtained
by glueing $E_{1,q_1}$ with $E_{2,q_2}$ along the isomorphism $\sigma$ (respectively along $\sigma'$), $\tau_i$ have to satisfy a compatibility
condition, summarized in the following commutative diagram:
$$
\xymatrix{
E_{1,q_1}\ar[r]^{\sigma}\ar[d]_{(\tau_1)_{q_1}} & E_{2,q_2}\ar[d]^{(\tau_2)_{q_2}}\\
E_{1,q_1}\ar[r]_{\sigma'} & E_{2,q_2}\\
}
$$

Since $E_i$ is stable we have $\Hom(E_i,E_i)\simeq \C\cdot \id_{E_i}$. Hence $(\tau_i)_{q_i}$ is the multiplication by some $\lambda_i\in\C^*$. In particular, $\sigma'$ is a non zero multiple of $\sigma$ and thus $[\sigma]=[\sigma']$.
\hfill\par
Now we prove that $\phi_{\vert {\mathscr U}}$ is a morphism. It is enough to 
 prove that $\phi$ is  regular at $u_0$, for any $u_0 \in \mathscr{U}$. 
 At this hand, we claim that there exists a non empty open subset 
 $W \subseteq  \mathscr{U} $  with $u_0 \in W$ and a vector bundle
${\mathcal E}$ on $W \times C$  such that 
$$[{\mathcal E}\vert_{ u \times C}] = \phi(u), \quad \forall u \in W.$$
{Step 1}: There exist two   sheaves  $\mathcal{Q}$ and ${\mathcal R}$ on $\mathscr{U} \times C$ such that, for each  $u = (([E_1],[E_2],[\sigma])$, with $u \in \mathscr{U}$,  we have
$${\mathcal Q}\vert_{ u \times C} \simeq {j_1}_*(E_1) \oplus {j_2}_*(E_2),
\quad  {\mathcal R}_{\vert u \times C} \simeq {j_p}_*({j_p}^*({j_2}_*(E_2))),$$
where $j_p \colon p \hookrightarrow C$
and  $j_i \colon C_i \hookrightarrow C$ are  the natural inclusions. 
\hfill\par
Consider the diagram 
\begin{equation}
\label{BIGDIAGRAM}
\xymatrix@C=3mm{
& \mathscr{U}\ar[dd]|\hole^<(0.25){\pi_{\mathscr{U}}} \ar[dl] & \mathscr{U}\times p \ar@/^1.8pc/@{_{(}->}[lld]_>(0.75){J_p}\ar[l]^{\simeq}\\
\mathscr{U}\times C\ar[dd]_{\Pi_{\mathscr{U}}} & & \\
& \mathcal{U}_{C_1}(r,d_1)\times \mathcal{U}_{C_2}(r,d_2)\ar[ld]\ar[r]^-{p_i} & \mathcal{U}_{C_i}(r,d_i)\ar[ld] & \mathcal{U}_{C_i}(r,d_i)\times q_i\ar[dl]^{\iota_i}\ar[l]_-{\simeq }\\
    \mathcal{U}_{C_1}(r,d_1)\times \mathcal{U}_{C_2}(r,d_2)\times C \ar@{->>}[r]_-{P_i}& \mathcal{U}_{C_i}(r,d_i)\times C & 
    \mathcal{U}_{C_i}(r,d_i)\times C_i\ar@{_{(}->}[l]^-{J_i}
}
\end{equation}

where the morphisms which appear have been defined as
\begin{equation}
J_i=\id_{\mathcal{U}_{C_i}(r,d_i)}\times j_i, \qquad 
P_i=p_i\times \id_C, \qquad 
\Pi_{\mathscr{U}}=\pi_\mathscr{U}\times \id_C, \quad J_p = \id_{\mathscr{U}}\times j_p.
\end{equation}
If, as before, we denote with $\mathcal{P}_i$ the Poincar\'e bundle on ${\mathcal U}_{C_i}(r,d_i) \times C_i$ we can set 
$$ {\mathcal Q_i} = \Pi_{\mathscr{U}}^*\Big({P_i}^*({J_i}_*({\mathcal P}_i))\Big) \quad  \mathcal Q = \mathcal Q_1 \oplus \mathcal Q_2,$$
and 
$$ {\mathcal R} = {J_p}_*({J_p}^*(Q_2)).$$
Note that $Supp({\mathcal R}) = \mathscr{U} \times p$. Moreover,  one can verify  that  if we identify $\mathscr{U} \times p $ with $\mathscr{U}$ we have:
\begin{equation}
\label{EQ:pullback}
J_p^*({\mathcal Q}_i) \simeq {\pi}_{\mathscr{U}}^*(p_i^*(\iota_i^*{\mathcal P}_i)),
\end{equation}
where $\iota_i \colon {\mathcal U}_{C_i}(r,d_i) \times q_i \hookrightarrow {\mathcal U}_{C_i}(r,d_i) \times C_i$. 

{Step 2}: There is an open subset $W \subset \mathscr{U}$ containing $u_0$ and a surjective map of sheaves
$$
\xymatrix{
\mathcal{Q}_1\oplus\mathcal{Q}_2|_{W\times C} \ar[r]^-{\Sigma_W} &
    \mathcal{R}|_{W\times C}
}
$$
whose kernel is the desired vector bundle $\mathcal{E}$ on $W\times C$.
\hfill\par

Let $\pi \colon {\mathbb P}({\mathcal F}) \to {\mathcal U}_{C_1}(r,d_1) \times {\mathcal U}_{C_2}(r,d_2)$ be the projective bundle defined in Lemma \ref{LEM:projbundle}.
First of all consider on ${\mathbb P}({\mathcal F})$ the tautological line bundle $\OO_{{\mathbb P}(\mathcal F)}(-1)$ which is, by definition, the subsheaf of $\pi^*({\mathcal F})$ whose fiber at $u \in {\mathbb P}({\mathcal F})$
is 
$$Span(\sigma) \subset \Hom(E_{1,q_1},E_{2,q_2}),$$
where  $u=(([E_1],[E_2]),[\sigma])$. 
We can choose $W$ to be an open subset of $\mathscr{U}$ containing the point $u_0$ and admitting a section $s \in \OO_{{\mathbb P}(\mathcal F)}(-1)(W)$ with 
$s(u)\not=0$, for any $u \in W$. 
\hfill\par

In particular $s$ induces a map of sheaves
\begin{equation}
    s  \colon {\pi}_{\mathscr{U}}^*p_1^*(\iota_1^*({\mathcal P}_1)))\vert_{W} \to {\pi}_{\mathscr{U}}^*p_2^*(\iota_2^*({\mathcal P}_2)))\vert_{ W}.
\end{equation}
 such that $s_u \colon E_{1,q_1} \to E_{2,q_2}$
is an isomorphism  and  $[s_u] = [\sigma]$ in ${\mathbb P}(\Hom(E_{1,q_1}, E_{2,q_2}))$.
\hfill\par

We can also define a morphism of sheaves 
\begin{equation}
\label{maps}
s-\id_2: {\pi}_{\mathscr{U}}^*p_1^*(\iota_1^*({\mathcal P}_1)))\vert_{W}\oplus {\pi}_{\mathscr{U}}^*p_2^*(\iota_2^*({\mathcal P}_2)))\vert_{W} \rightarrow {\pi}_{\mathscr{U}}^*p_2^*(\iota_2^*({\mathcal P}_2)))\vert_{W}
\end{equation}
where $\id_2$ is the identity of ${\pi}_{\mathscr{U}}^*p_2^*(\iota_2^*({\mathcal P}_2)))\vert_{W}$. 
\hfill\par

This allows us to define the map $\Sigma_W$ we are looking for. Indeed, since $Supp({\mathcal R}\vert_{ W \times C})= W \times p$, it  is enough  to give the  map on $W\times p$, which can be identified with   $W$.  Using the isomorphism \ref{EQ:pullback},   we have a diagram which defines $\Sigma_W$: 
$$
\xymatrix{
\mathcal{Q}_1\oplus\mathcal{Q}_2|_{W\times C} \ar[r]^-{\Sigma_W}\ar[d]_-{|_{W\times p}} &
    \mathcal{R}|_{W\times C}\ar[d]^-{|_{W\times p}}\\
J_p^*(\mathcal{Q}_1\oplus\mathcal{Q}_2|_{W\times C}) \ar[r]^-{\Sigma_{W}|_{W\times p}}\ar[d]_-{\simeq } & 
    J_p^*(\mathcal{R}|_{W\times C})\ar[d]^-{\simeq }\\
{\pi}_{\mathscr{U}}^*p_1^*(\iota_1^*({\mathcal P}_1)))\vert_{W}\oplus{\pi}_{\mathscr{U}}^*p_2^*(\iota_2^*({\mathcal P}_2)))\vert_{W} \ar[r]_-{s-\id_2} & 
    {\pi}_{\mathscr{U}}^*p_2^*(\iota_2^*({\mathcal P}_2)))\vert_{W}
}$$
By taking the kernel $\mathcal{E}$ of this map we concludes the second step of the proof of the claim. 
In particular, $\phi_{\vert {\mathscr U}}$ is a morphism.
\hfill\par

By construction, $\phi(\mathscr{U})$ is contained in ${\mathcal V}_C(w,r,\chi)_{d_1,d_2}$ and it coincide with the open subset of $w$-semistable vector bundles whose restrictions are semistable.  Moreover, ${\mathcal V}_C(w,r,\chi)_{d_1,d_2}$ is a dense open subset of ${\mathcal U}_C(w,r,\chi)_{d_1,d_2}$ (see \cite{S}). 
Notice that, by \ref{REM:dim}, we have: 
$$\dim (\phi(\mathscr{U}))=\dim(\mathscr{U}) = r^2(g_1+g_2-1) +1$$
 which is  the dimension of 
$\mathcal{U}_C(w,r,\chi)_{d_1,d_2}$, see 
Theorem \ref{modulispace}. 
This implies that $\phi$ is a dominant map. Hence, by a generic smoothness argument, we can conclude that $\phi_{\vert {\mathscr U}}$ is a birational morphism.
\end{proof}

\begin{corollary}
Let $C$ be a nodal curve with a single node  $p$ and two smooth irreducible components $C_i$ of genus $g_i \geq 1$. Assume  that the moduli space $\mathcal{U}_C(w,r,\chi)$ has an irreducible component corresponding to bidegree $(d_1,d_2)$ with  $d_1$ and $d_2$ coprime with $r$.  Then this component is birational to a projective bundle over the smooth variety
$\mathcal{U}_{C_1}(r,d_1)\times \mathcal{U}_{C_2}(r,d_2)$.
\end{corollary}

Note that $\phi$ provides a desingularization of the component ${\mathcal U}_C(w,r,\chi)_{d_1,d_2}$. 

If the genus of the curve $C_i$ is big enough, we can be more precise about the domain of the rational map $\phi$.
\hfill\par
Assume that $g_i > r+1$, then   by Lemma \ref{LEM:OPENSTABLE}
  the locus of vector bundles of ${\mathcal U}_{C_i}(r,d_i)$ which are $(0,r)$-stable is a non empty open subset of ${\mathcal U}_{C_i}(r,d_i)$, let us denote it by $V_i$.
\begin{definition} We will denote by ${\mathscr V}$ the open subset  $\pi^{-1}(V_1 \times V_2)$  in ${\mathbb P}({\mathcal F})$.
\end{definition}
By construction,
$\mathscr{V} $ is a projective bundle over  $ V_1 \times V_2$.

\begin{theorem}
\label{THM:2}
Assume that the hypothesis of Theorem \ref{THM:MAIN} hold. Moreover, assume that $g_i > r+1$ and $(\chi_1, \chi_2) \in {\mathcal W}_{r}$. Then there exists a polarization $w$ such that the map $\phi$ sending $u$ to $[E_u]$ is a birational map such that ${\phi}\vert_{\mathscr{U}\cup \mathscr{V}}$ is a morphism. 
\end{theorem}
\begin{proof}
Since  $(\chi_1,\chi_2) \in {\mathcal W}_{r}$ then, 
by Remark \ref{allk}, there exists  a polarization $w$
such that conditions \ref{numconditions} hold for any  $k = 1, \cdots ,r$.  In particular, as ${\mathcal W}_r \subset {\mathcal W}_{r,r}$,  Theorem  \ref{THM:MAIN} holds:  $\phi$ is a birational map which  is defined on the open subset ${\mathscr U}$.
\hfill\par
Assume that $u \in {\mathscr V}$ and $u \not\in {\mathcal U}$. Then $u = (([E_1],[E_2]),[\sigma])$, with
$([E_1],[E_2]) \in V_1 \times V_2$ and $\Rk \sigma \leq r-1$.
Since   $[E_i] \in V_i$, then by lemma \ref{PROP:STABILITY}, $E_u$ is $w$-semistable, hence $\phi$ is defined all over the open subset $\mathscr V$ too.
\hfill\par
To prove that  $ {\phi}\vert_{\mathscr{V}}$ is a  morphism, we can proceed as in the proof of Theorem \ref{THM:MAIN}, just by replacing ${\mathscr U}$ with 
${\mathscr V}$ and ${\mathcal U}_{C_i}(r,d_i)$ with 
$V_i$. 
\end{proof}

\section{Fixed-determinant moduli space}
\label{sec5}

Let $C$ be a smooth curve of genus $g \geq 1$ and $L \in \Pic^{d}(C)$.  We recall that the moduli space of semistable vector bundles of rank $r$ and determinant $L$  on $C$ is denoted by ${\mathcal SU}_{C}(r,L)$ and it is an irreducible and projective variety. 
It is the fiber of the determinant map:
$$ \det \colon {\mathcal U}_C(r,d) \to \Pic^d(C).$$
In this section we will investigate a similar subvariety of the moduli space ${\mathcal U}_C(w,r,\chi)_{d_1,d_2}$ for a nodal reducible curve with two irreducible component $C_i$. 
Fix  a pair  $(L_1,L_2)$ with  $L_i \in \Pic^{d_i}(C_i)$.  Note that there exists a unique line bundle $L$ on the nodal curve $C$ whose restriction to the component $C_i$ is $L_i$. Recall that ${\mathcal V}_C(w,r,\chi)_{d_1,d_2} \subset {\mathcal U}_C(w,r,\chi)_{d_1,d_2}$ is the open subset parametrizing $w$-semistable classes which are represented by vector bundles.   

\begin{definition}
Let $L$ be the line bundle on $C$ induced by the pair 
$(L_1,L_2)$.  We define ${\mathcal SU}_C(w,r,L) $ as the closure of 
$$\{ [E] \in {\mathcal V}_C(w,r,\chi)_{d_1,d_2}  \, \vert \, \det E = L \}$$
in ${\mathcal U}_C(w,r,\chi)_{d_1,d_2}$.
\end{definition}

If we assume that $r$ and $d_i$ are coprime,  then ${\mathcal SU}_{C_i}(r,L_i)$  is a  smooth irreducible  projective variety of dimension $(r^2-1)(g_i-1)$.  
As in Lemma \ref{LEM:projbundle}, we can 
define a vector  bundle  ${\mathcal F}_L$ on 
${\mathcal SU}_{C_1}(r,L_1) \times {\mathcal SU}_{C_2}(r,L_2)$ just by restricting ${\mathcal F}$. Then we can consider the associated projective bundle  ${\mathbb P}({\mathcal F}_L)$ and
$$\mathscr{U}_L = \mathscr{U} \cap {\mathbb P}({\mathcal F}_L),$$
a $\PGL(r)$-bundle on ${\mathcal SU}_{C_1}(r,L_1) \times {\mathcal SU}_{C_2}(r,L_2)$. We denote by $\phi_L$ the restriction of the morphism $\phi$ defined in Theorem \ref{THM:MAIN} to $\mathscr{U}_L$.
As a consequences of Theorem \ref{THM:MAIN},  we have the following:
\begin{corollary}
\label{COR1}
In the hypothesis of Theorem \ref{THM:MAIN}, the 
 map 
$$
\xymatrix{
\phi_L \colon {\mathbb P}({\mathcal F}_L) \ar@{-->}[r] & {\mathcal SU}_C(w,r,L) 
}$$
is a birational map, whose restriction ${\phi}_L\vert_{{\mathscr U}_L}$ is an injective morphism. 
\end{corollary}
\begin{proof}
Note that ${\phi}_L\vert_{{\mathscr U}_L}$ is  a morphism and its image is the following subset:
$$Im \phi_L = \{ E \in \ {\mathcal V}_C(w,r,\chi)_{d_1,d_2} \ \vert \  {[E_{\vert C_i}]} \in {\mathcal SU}_{C_i}(r,L_i) \}.$$
In particular, $Im \phi_L \subseteq {\mathcal SU}_C(w,r,L)$.
\hfill\par
Consider the following map:
$$\psi \colon {\mathcal V}_C(w,r,\chi)_{d_1,d_2}  \to \Pic^{d_1}(C_1) \times \Pic^{d_2}(C_2),$$
sending $E \to (\det(E\vert_{ C_1}), \det (E\vert_{ C_2}))$, which fit into the following commutative diagramm:
\begin{equation}
\xymatrix{
{\mathscr U} \ar[r]^{\phi} \ar[d]_{{\pi}_{\mathscr U}} & {\mathcal V}_C(w,r,\chi)_{d_1,d_2}  \ar[d]^{\psi}\\
{\mathcal U}_{C_1}(r,d_1) \times {\mathcal U}_{C_2}(r,d_2) \ar[r]_{\det_1 \times  \det_2} & \Pic^{d_1}(C_1) \times \Pic^{d_2}(C_2)\\
}
\end{equation}
It follows immediately that $\psi$ is a surjective morphism and $Im {\phi}_L \subset \psi^{-1}(L_1,L_2)$.
\hfill\par
We claim that $\psi$ has irreducible fibers of dimension  $(r^2-1)(g_1+g_2-1)$. 
\hfill\par
First of all we prove that any two fibers of $\psi$ are isomorphic.  Let  $(L_1,L_2) $ and $(L_1',L_2') $  in $\Pic^{d_1}(C_1) \times \Pic^{d_2}(C_2)$, then there exists $\xi_i \in \Pic^0(C_i) $ such that
$L_i \otimes \xi_i^{r} \simeq L_i'$. 
Let $\xi$ be the unique line bundle on $C$ such that $\xi_{\vert C_i} \simeq \xi_i$.
The natural map
$$ \psi^{-1}(L_1,L_2) \to \psi^{-1}(L_1',L_2')$$
sending $E$ to $E \otimes \xi$  preserves  $w$-semistability and  it  gives  an  isomorphism of the fibers. In particular, from fiber dimension Theorem (see \cite{H}, p.95), this  implies that    any fiber has pure dimension  $(r^2-1)(g_1+g_2-1)$. 
\hfill\par
Finally  we prove that any fiber is irreducible. Let $Y = {\mathcal V}_C(w,r,\chi)_{d_1,d_2} \setminus \phi(\mathscr{U})$, it is a proper subvariety of ${\mathcal V}_C(w,r,\chi)_{d_1,d_2} $.
Assume that the fiber of $\psi$ over $(L_1,L_2)$ is reducible.  Let $F_1$ be the irreducible component containing   $\phi(\mathscr{U}_L)$, then there exists an irreducible component $F_2 \subset Y$.
So the restriction of $\psi$ to $Y$ is a  surjective morphism whose fibers have dimension $(r^2-1)(g_1+g_2-1)$. This implies that $\dim Y = \dim {\mathcal V}_C(w,r,\chi)_{d_1,d_2} $, which is impossible. 
\hfill\par
This allows us to conclude that ${\mathcal SU}_C(w,r,L)$ is irreducible too and 
$\phi_L$ is a birational morphism. 
\end{proof}

\begin{theorem}
\label{THM2}
In the hypothesis of Theorem \ref{THM:MAIN},
${\mathcal SU}_C(w,r,L) $ is a rational variety. 
\end{theorem}

\begin{proof}
By hypothesis $d_i$ and $r$  are coprime, then the moduli space  ${\mathcal SU}_{C_i}(r,L_i)$ is rational for any line bundle $L_i \in \Pic^{d_i}(C_i)$, see \cite{KS}, \cite{N1} and \cite{N2}.
Since $\mathscr{U}_L$ is a ${\mathbb P}^{r^2-1}$-bundle over the product 
${\mathcal SU}_{C_1}(r, L_1) \times {\mathcal SU}_{C_2}(r,L_2)$, then it is a rational variety too. 
The assertion follows from  corollary \ref{COR1}. 
\end{proof}

\end{document}